\documentclass[reqno,11pt]{amsart}
\usepackage{a4}
\usepackage{latexsym}
\usepackage{fullpage}
\usepackage{amssymb}
\usepackage{amscd}
\usepackage{hyperref}
\numberwithin{equation}{section}
\swapnumbers
\theoremstyle{plain}
\newtheorem{Lemma}[equation]{Lemma}
\newtheorem{Proposition}[equation]{Proposition}
\newtheorem{Corollary}[equation]{Corollary}

\newtheorem{Theorem}[equation]{Theorem}

\theoremstyle{definition}
\newtheorem{Definition}[equation]{Definition}

\theoremstyle{remark}
\newtheorem*{rem}{Remark}
\newtheorem{Remark}[equation]{Remark}

\newcommand\R{\ensuremath{\mathbb R}}
\newcommand\C{\ensuremath{\mathbb C}}
\newcommand\Q{\ensuremath{\mathbb Q}}
\newcommand\Z{\ensuremath{\mathbb Z}}

\newcommand\B{\ensuremath{\mathcal B}}
\renewcommand\P{\ensuremath{\mathcal P}}
\newcommand\Bhat{\widehat{\B}}
\newcommand\Phat{\widehat{\P}}
\newcommand\FThat{\widehat{\FT}}
\newcommand\SL{SL}
\newcommand\FT{\mathrm{FT}}

\newcommand\Lhat{\ensuremath{\widehat L}}
\newcommand\Chat{\ensuremath{\widehat{\mathbb C}}}
\newcommand\chihat{\ensuremath{\widehat\chi}}
\newcommand\chat{\ensuremath{\hat c}}
\newcommand\kaphat{\ensuremath{\widehat\kappa}}
\newcommand\lamhat{\ensuremath{\widehat\lambda}}
\newcommand\pihat{\ensuremath{\widehat\pi}}

\newcommand\Li{\operatorname{Li}}
\newcommand\Log{\operatorname{Log}}
\newcommand\Arg{\operatorname{Arg}}
\renewcommand\Re{\operatorname{Re}}
\renewcommand\Im{\operatorname{Im}}
\newcommand\punkt{\mathord{\,\cdot\,}}
\newcommand\im{\operatorname{im}}

\newcommand\cut{{\mathrm{cut}}}

\def\abs#1{\left|#1\right|}
\let\<\langle
\let\>\rangle
\newcommand\Hom{\operatorname{Hom}}

\begin{document}

\title{The Extended Bloch Group\\and the Cheeger-Chern-Simons Class}
\author{Sebastian Goette}
\address{Mathematisches Institut, Universit\"at Freiburg,
79104 Freiburg, Germany}
\email{sebastian.goette@math.uni-freiburg.de}
\thanks{First author supported in part by DFG special programme
``Global Differential Geometry''}
\author{Christian Zickert}
\address{Department of Mathematics, Columbia University, New York,
NY 10027, USA}
\email{zickert@math.columbia.edu}
\thanks{Second author supported by ``Rejselegat for matematikere''}
\begin{abstract} 
  We present a formula for the full Cheeger-Chern-Simons class
  of the tautological flat complex vector bundle of rank~$2$
  over~$B\SL(2,\C^\delta)$.
  Our formula improves the formula in [DZ],
  where the class is only computed modulo $2$-torsion. 
\end{abstract}

\maketitle

\section*{Introduction}
The Cheeger-Chern-Simons class~$\chat_k$ is a natural refinement of
the $k$-th Chern class for complex vector bundles with connection,
and it takes values in the ring of differential characters,
see~\cite{CS1}, \cite{CS2}.
For a vector bundle with a flat connection,
this class becomes an ordinary $(2k-1)$-cohomology class
with coefficients in~$\C/\Z(k)$, where~$\Z(k)=(2\pi i)^k\,\Z$.
Let~$B\SL(n,\C^\delta)$ denote the classifying space
of the group~$\SL(n,\C^\delta)$
with the discrete topology.
The universal Cheeger-Chern-Simons
class~$\chat_k\in H^{2k-1}(B\SL(n,\C^\delta),\C/\Z(k))$ of the tautological
flat complex vector bundle over~$B\SL(n,\C^\delta)$
gives rise to the Borel regulator in algebraic $K$-theory,
and~$\chat_2$ is also related to invariants of hyperbolic $3$-manifolds.
One is interested in a combinatorial description
of this
class~$\chat_k\in H^{2k-1}(B\SL(n,\C^\delta),\C/\Z(k))$.
Dupont derived an expression
for~$\chat_2\in H^3(BSL(2,\C^\delta),\C/\Z(2))$
modulo~$\Q(2)$ in~\cite{Du}.
A similar formula for~$\Re\hat c_3$ is due to Goncharov,
see~\cite{Goncharov}.

The homology of the classifying space of a discrete group is by definition
the homology of the group, and since~$\C/\Z(2)$ is divisible  we can regard~$\chat_2$ as a homomorphism~$H_3(SL(2,\C))\to \C/\Z(2)$. The natural map~$H_3(SL(2,\C))\to H_3(PSL(2,\C))$ has cyclic kernel of order~$4$, so we have a commutative diagram defining~$\chat_2$ on~$H_3(PSL(2,\C))$,
\begin{equation*}
  \begin{CD}
    H_3(\SL(2,\C^\delta))@>\chat_2>>\C/\Z(2)\\
    @VVV@VVV\\
    H_3(P\SL(2,\C^\delta))@>\chat_2>>\C/\pi^2\Z&\;.
  \end{CD}
\end{equation*}
An explicit formula for the lower map was obtained in~\cite{N},
and in~\cite{DZ} this was extended to a formula for the upper map.
However, the formula given in~\cite{DZ} only computes the image of~$\chat_2$
in~$\C/2\pi^2\Z$, thus only computing~$\chat_2$ up to $2$-torsion
(see Remark 4.2 in~\cite{DZ} for a comment on the normalization).
In the present paper we extend the result in~\cite{DZ} obtaining
a formula computing the full Cheeger-Chern-Simons class. 

In both~\cite{N} and~\cite{DZ} the formulas are obtained by factoring~$\chat_2$
through a version of the \emph{extended Bloch group},
an object defined by Neumann in~\cite{N}.
There are two different versions of the extended Bloch group.
One version, denoted~$\widehat \B(\C)$ in Neumann's paper,
is generated by symbols~$[z;p,q]$ subject to a five term relation \emph{and}
a transfer relation.
It is isomorphic to~$H_3(P\SL(2,\C^\delta))$.
The other one, denoted~$\mathcal E\B(\C)$ is generated by symbols~$[z;2p,2q]$
and \emph{only} subject to the five term relation.
The latter version is called the \emph{more extended Bloch group},
and conjectured to be isomorphic to~$H_3(\SL(2,\C^\delta))$ in~\cite{N}.
The role of the transfer relation is subtle and has caused some minor
inaccuracies in~\cite{N} and~\cite{DZ},
see Remark~\ref{TransferRemark} below.
Proposition~8.2 and Corollary~8.3 in~\cite{N} are only correct
if we include the transfer relation.
Proposition~8.2 has been used in the proof of Proposition~4.15 in~\cite{DZ},
so this result is also only correct if we include the transfer relation.
To the best of our knowledge these are the only problems in~\cite{N}
and~\cite{DZ}.
We present corrections to these results as Theorem~\ref{ChihatLemma},
Corollaries~\ref{ExtBlochCor} and~\ref{MainCor},
and Remark~\ref{PSLRem} below. 

To obtain the full class~$\chat_2$,
we construct a lift
of the function~$\widehat L\colon\Phat(\C)\to\C/2\pi^2\Z$
defined in~\cite{DZ}
with values in~$\C/4\pi^2\Z=\C/\Z(2)$.
Note that this lift is not compatible with the transfer relation,
see~\ref{TransferRemark}.
This observation was the main motivation for the present note.

The paper is organized as follows:
In Section~\ref{ExtBlochSect} we recall the definition of the extended Bloch
group.
In Section~\ref{RogersSect} we define the modified extended Rogers
dilogarithm.
In Section~\ref{RelSect} we work out some relations
in the extended Bloch group,
and in Section~\ref{CCSSect} we prove our main results that the extended
Bloch group is isomorphic to~$H_3(\SL(2,\C^\delta))$ as conjectured
in~\cite{N},
and that~$\Lhat$
computes the Cheeger-Chern-Simons class.
In Section~\ref{MoreSect},
we have added some more relations in~$\Phat(\C)$
that might be of interest elsewhere.

\begin{rem}
  In the present paper we denote the more extended Bloch
  group~$\Bhat(\C)$ instead of~$\mathcal E\B(\C)$
  and refer to it as the extended Bloch group.
  This is consistent with the notation in~\cite{DZ}.
  Neumann's original extended Bloch group will be named~$\Bhat_N(\C)$.
\end{rem}

We are grateful to J. Dupont, G. Kings and W. Neumann
for fruitful discussions and their interest in our work.
Parts of this note were written
while the  first named author enjoyed the hospitality of the Chern Institute
at Tianjin.

\section{The Extended Bloch Group}\label{ExtBlochSect}
We follow the description in~\cite{N} at the beginning of chapter~8.
%
We first recall the construction of the classical Bloch group.
Define a set of five term relations
\begin{equation}\label{FTdef}
  \FT=\biggl\{\,\biggl(x,y,\frac yx,\frac{1-1/x}{1-1/y},\frac{1-x}{1-y}\biggr)
  \biggm|x\ne y\in\C\setminus\{0,1\}\,\biggr\}\subset(\C\setminus\{0,1\})^5\;.
\end{equation}
Consider the free abelian groups
generated by the elements of~$\FT$ and~$\C\setminus\{0,1\}$
and the chain complex
\begin{equation}\label{BlochComplex}
  \begin{CD}
    \Z[\FT]@>\rho>>\Z[\C\setminus\{0,1\}]@>\nu>>\C^\times\wedge_\Z\C^\times
  \end{CD}
\end{equation}
with arrows defined on generators by
\begin{align*}
  \rho([z_0,\dots,z_4])&=[z_0]-[z_1]+[z_2]-[z_3]+[z_4]\;,\\
  \nu([z])&=z\wedge(1-z)^{-1}\;.
\end{align*}
Then~$\P(\C)=\Z[\C\setminus\{0,1\}]/\im\rho$
is called the {\em pre-Bloch group,\/}
and the middle homology of the complex~\eqref{BlochComplex} above
is called the {\em Bloch group\/}~$\B(\C)$.

Let~$\Chat$ denote the universal abelian cover of~$\C\setminus\{0,1\}$.
To construct~$\Chat$,
we start with~$\C_\cut=\C\setminus((-\infty,0]\cup[1,\infty))$.
For each~$x\in(-\infty,0)\cup(1,\infty)$,
we add two boundary points
	$$x\pm 0i:=\lim_{t\searrow 0}x\pm ti$$
and put
	$$\overline\C_\cut=\C_\cut
		\cup\{\,x\pm 0i\mid x\in(-\infty,0)\cup(1,\infty)\,\}\;.$$
We extend the principal branches of~$\Log$,
$\Li_2$ and~$\Arg=\Im\Log$ to~$\overline\C_\cut$.

In~$\overline\C_\cut\times(2\Z)^2$,
identify
\begin{align}
  \begin{split}\label{ChatDef}
    (x+0i,2p,2q)&\sim(x-0i,2p+2,2q)
	\qquad\text{for all~$x\in(-\infty,0)$, and}\\
    (x+0i,2p,2q)&\sim(x-0i,2p,2q+2)
	\qquad\text{for all }x\in(1,\infty)
  \end{split}
\end{align}
for all~$p$, $q\in\Z$,
obtaining~$\pihat\colon\Chat\to\C\setminus\{0,1\}$.
Equivalence classes will be denoted~$(x;2p,2q)$ or simply~$\widehat x$.
Note that we could have dropped the factors~$2$ above
and worked in~$\overline\C_\cut\times\Z^2$ instead,
but we want to stay compatible with~\cite{DZ} and~\cite{N}.

As in~\cite{N}, put
\begin{equation*}
  \FT^+=\bigl\{\,(z_0,\dots,z_4)\in \FT\bigm|\Im z_0,\dots,\Im z_4>0\,\}\;.
\end{equation*}
Then~$(z_0,\dots,z_4)\in \FT^+$ iff~$\Im z_1>0$
and~$z_0$ is in the interior of the Euclidean triangle
spanned by~$0$, $1$ and~$z_1$.
Let~$\FThat$ denote the connected component of~$\pihat^{-1}(\FT)\subset\Chat^5$
that contains
	$$\FThat^+=\bigl\{\,((z_0;0,0),\dots,(z_4;0,0))\bigm|
		(z_0,\dots,z_4)\in \FT^+\,\bigr\}\;.$$

Also, note that the functions
	$$(z;2p,2q)\longmapsto(\Log z+2\pi i\,p)
		\qquad\text{and}\qquad
	(z;2p,2q)\longmapsto(-\Log(1-z)+2\pi i\,q)$$
are holomorphic on~$\Chat$.
By~\cite{DZ},
we can extend~\eqref{BlochComplex} to a chain complex
\begin{equation}\label{ExtBlochComplex}
  \begin{CD}
    \Z\bigl[\FThat\bigr]@>\widehat\rho>>\Z\bigl[\Chat\bigr]
    @>\widehat\nu>>\C\wedge_\Z\C
  \end{CD}
\end{equation}
with arrows defined on generators by
\begin{align*}
  \widehat\rho([\widehat z_0,\dots,\widehat z_4])
  &=[\widehat z_0]-[\widehat z_1]+[\widehat z_2]-[\widehat z_3]+[\widehat z_4]\;,\\
  \widehat\nu([z;2p,2q])
  &=(\Log z+2\pi i\,p)\wedge(-\Log(1-z)+2\pi i\,q)\;.
\end{align*}

\begin{Definition}\label{ExtBlochDef}
  The {\em extended pre-Bloch group\/}~$\Phat(\C)$
  is defined as the quotient~$\Z[\Chat]/\im\widehat\rho$,
  and the {\em extended Bloch group\/}~$\Bhat(\C)$
  as the middle homology of the complex~\eqref{ExtBlochComplex}.
\end{Definition}

\section{The Extended Rogers Dilogarithm}\label{RogersSect}
The classical dilogarithm is given by
	$$\Li_2(z)=\sum_{k=1}^\infty\frac{z^k}{k^2}
	=-\int_0^z\log(1-t)\,\frac{dt}t$$
for all~$z$ with~$\abs z<1$.
It extends to a multivalued function on~$\C$
with branch points at~$0$, $1$ and~$\infty$.
Recall that the Rogers dilogarithm~$L\colon(0,1)\to\R$ is given by
	$$L(x)=\Li_2(x)+\frac12\,\log x\,\log(1-x)-\frac{\pi^2}6\;.$$
We extend~$L$ to a holomorphic function
\begin{equation}\label{LhatDef}
  \begin{gathered}
    \overline L\colon\overline\C_\cut\times(2\Z)^2\longrightarrow\C\;,\\
    \overline L(z;2p,2q)=\Li_2(z)+\frac12\,\bigl(\Log z+2\pi i\,p\bigr)
	\,\bigl(\Log(1-z)+2\pi i\,q\bigr)-\frac{\pi^2}6\;.
  \end{gathered}
\end{equation}

\begin{Lemma}\label{RogersLemma}
  The function~$\overline L$ above induces a holomorphic function
	$$\Lhat\colon\Chat\to\C/\Z(2)$$
  that satisfies the five term relation
	$$\sum_{k=0}^4(-1)^k\Lhat(\widehat z_k)=0$$
  for all~$(\widehat z_0,\dots,\widehat z_4)\in\FThat$.
\end{Lemma}

Note that~$\Lhat$ lifts the function~$\Lhat$ in~\cite{DZ}
from~$\C/2\pi^2\Z$ to~$\C/4\pi^2\Z=\C/\Z(2)$.

\begin{proof}
  Because~$\Li_2(z)$ has no singularity at~$z=0$
  and~$(z;p,q)\mapsto\Log z+2\pi i p$ is holomorphic on~$\Chat$,
  the function~$\overline L$ extends holomorphically
  across~$(-\infty,0)\times(2\Z)^2$ in~$\Chat$.
  
  If we extend~$\Li_2$ and~$\Log(1-z)$ to~$z\pm 0i$ for~$z>1$, then
  \begin{align*}
    \Li_2(z+0i)&=\Li_2(z-0i)+2\pi i\,\Log z\\
    \Log(1-(z+0i))
    &=\Log(1-(z-0i))-2\pi i\;.
  \end{align*}
  Hence
  \begin{align*}
    \overline L(z+0i;2p,2q)
    &=\overline L(z-0i;2p,2q+2)+4\pi^2\,p\;,
  \end{align*}
  so the extension~$\Lhat$ of~$\overline L$ mod~$4\pi^2$ is well-defined.

  By~\cite{N}, we have a five term relation
	$$\sum_{k=0}^4(-1)^k\overline L(\widehat z_k)=0\in\C$$
  for all~$(\widehat z_0,\dots,\widehat z_4)\in\FThat^+$.
  Because~$\FThat$ is a connected complex manifold,
  the five term relation for~$\Lhat$
  holds in~$\C/\Z(2)$ for all~$(\widehat z_0,\dots,\widehat z_4)\in\FThat$
  by analytic continuation.
\end{proof}

\begin{Remark}
  Along the commutator of a small loop
  around~$0$ and a small loop around~$1$ in~$\C\setminus\{0,1\}$,
  the holomorphic continuation of the Rogers dilogarithm
  changes by~$4\pi^2$.
  This shows that we cannot lift~$L$ to a holomorphic function on~$\Chat$
  with values in~$\C/A$, with~$A\subset\Z(2)$ a proper subgroup.
\end{Remark}

\begin{Corollary}
  The function~$\Lhat$ induces homomorphisms
	$$\Lhat\colon\Phat(\C)\to\C/\Z(2)
		\qquad\text{and}\qquad
	\Lhat\colon\Bhat(\C)\to\C/\Z(2)\;.$$
\end{Corollary}

\section{Relations in the Extended Bloch Group}\label{RelSect}
Following~\cite{N},
we find relations among elements of the extended pre-Bloch group.
We then parametrize the kernel of the forgetful maps~$\Phat(\C)\to\P(\C)$
and~$\Bhat(\C)\to\B(\C)$ induced by the projection~$\pihat\colon\Chat\to\C$.
In the following,
we will identify~$z\in(-\infty,0)\cup(1,\infty)$ with~$z+0i$.

As explained in the appendix of~\cite{DZ},
we have
\begin{multline}\label{FThatEq}
  \pihat^{-1}(\FT^+)\cap\FThat
  =\Bigl\{\,\Bigl(\bigl(z_0;2p_0,2q_0\bigr),\bigl(z_1;2p_1,2q_1\bigr),
	\bigl(z_2;2(p_1-p_0),2q_2\bigr),\\
	\bigl(z_3;2(p_1-p_0+q_1-q_0),2(q_2-q_1)\bigr),
	\bigl(z_4;2(q_1-q_0),2(q_2-q_1-p_0)\bigr)\Bigr)\\
  \Bigm|(z_0,\dots,z_4)\in \FT^+\text{ and }
	p_0,p_1,q_0,q_1,q_2\in\Z\,\Bigr\}\;.
\end{multline}
For other choices of~$(z_0,\dots,z_4)\in \FT$,
a few of the~$p_k$, $q_k$ have to be adjusted by~$\pm 1$.

Subtracting two instances of the five term relation and using~\eqref{FTdef}
and~\eqref{FThatEq},
we obtain Neumann's cycle relation
\begin{multline*}
  [x;2p_0,2q_0-2]-[y;2p_1,2q_1-2]
	+\Bigl[\frac yx;2p_1-2p_0,2q_2-2\Bigr]\\
  =[x;2p_0,2q_0]-[y;2p_1,2q_1]+\Bigl[\frac yx;2p_1-2p_0,2q_2\Bigr]
\end{multline*}
for all~$x$, $y$
such that~$\bigl(x,y,\frac yx,\frac{1-1/x}{1-1/y},\frac{1-x}{1-y}\bigr)
\in \FT^+$.
If we vary~$x$ and~$y$ continuously,
then some of the integers in this relation may jump.
Thus, we obtain
\begin{multline}\label{CycleRel}
  [x;2p_0,2q_0-2]-[x;2p_0,2q_0]-[y;2p_1,2q_1-2]+[y;2p_1,2q_1]\\
  =
  \begin{cases}
    \bigl[\frac yx;2p_1-2p_0-2,2q_2-2\bigr]
	-\bigl[\frac yx;2p_1-2p_0-2,2q_2\bigr]
		&\text{if }\Arg y-\Arg x\le-\pi\;,\\
    \bigl[\frac yx;2p_1-2p_0,2q_2-2\bigr]
	-\bigl[\frac yx;2p_1-2p_0,2q_2\bigr]
		&\text{if }-\pi<\Arg y-\Arg x\le\pi\;,\\
    \bigl[\frac yx;2p_1-2p_0+2,2q_2-2\bigr]
	-\bigl[\frac yx;2p_1-2p_0+2,2q_2\bigr]
		&\text{if }\pi<\Arg y-\Arg x\;.\\
  \end{cases}
\end{multline}

Subtracting two instances of~\eqref{CycleRel}
gives
\begin{equation}\label{QRel}
  [x;2p,2(q-1)]-[x;2p,2q]=[x;2p,2(q'-1)]-[x;2p,2q']
\end{equation}
for all~$x\in\overline\C_\cut$.
Similarly,
one can prove
\begin{align}\label{PRel}
  [x;2(p-1),2q]-[x;2p,2q]&=[x;2(p'-1),2q]-[x;2p',2q]\\\label{PQRel}
	\text{and}\qquad
  [x;2(p+1),2(q-1)]-[x;2p,2q]&=[x;2(p'+1),2(q'-1)]-[x;2p',2q']
\end{align}
for all~$x\in\overline\C_\cut$ and all~$p$, $q$, $p'$, $q'\in\Z$
such that~$p'+q'=p+q$.
See~\cite{N} for a more geometric derivation of these relations.

By Lemma~7.3 in~\cite{N},
we also have the relation
\begin{equation}\label{MirrorRel}
  [z;2p,2q]+[1-z;-2q,-2p]=2\,\biggl[\frac12;0,0\biggr]\;,
\end{equation}
which is of course compatible with~\eqref{QRel}--\eqref{PQRel}.

Let us define
	$$\{z;2p\}=[z;2p,2q]-[z;2p,2(q-1)]\;,$$
which is independent of~$q$ by~\eqref{QRel}.

\begin{Lemma}\label{HomoLemma}
  For all~$z$, $w\in\overline\C_\cut$ with~$zw\ne 1$,
  we have the relation
	$$\{z;2p\}+\{w;2r\}=
	\begin{cases}
	  \{zw+0i;2(p+r-1)\}	&\text{if~$\Arg z+\Arg w\le-\pi$,}\\
	  \{zw+0i;2(p+r)\}	&\text{if~$-\pi<\Arg z+\Arg w\le\pi$, and}\\
	  \{zw+0i;2(p+r+1)\}	&\text{if~$\pi<\Arg z+\Arg w$.}
	\end{cases}$$
\end{Lemma}

\begin{proof}
  This is immediate from~\eqref{CycleRel}.
\end{proof}

\begin{Lemma}\label{KappaLemma}
  The element~$\kaphat=\{z;2p\}-\{z;2(p-1)\}\in\Phat(\C)$
  is independent of~$z\in\C\setminus\{0,1\}$ and~$p\in\Z$,
  and of order~$2$ in~$\Phat(\C)$.
\end{Lemma}

\begin{proof}
  Independence of~$p$ follows from~\eqref{PRel} and
  independence of~$z$ is immediate from Lemma~\ref{HomoLemma}.
  To proof that~$\kaphat$ is of order two,
  use~\eqref{QRel} and~\eqref{MirrorRel} to write
  \begin{align*}
    \kaphat
    &=[z;2,0]-[z;2,-2]-[z;0,0]+[z;0,-2]\\
    &=-[1-z;0,-2]+[1-z;2,-2]+[1-z;0,0]-[1-z;2,0]\\
    &=\{1-z;0\}-\{1-z;2\}=-\kaphat\;.
  \end{align*}
  To show that~$\kaphat\ne0$, we compute
  \begin{equation}\label{LhatKaphatFormel}
    \Lhat(\kaphat)
    =\Lhat(z;2,2)-\Lhat(z;2,0)-\Lhat(z;0,2)+\Lhat(z;0,0)
    =-2\pi^2\;,
  \end{equation}
  and~$-2\pi^2\ne0$ in~$\C/\Z(2)$.
\end{proof}

\begin{Remark}\label{TransferRemark}
  In~\cite{N}, Section~8, and~\cite{DZ}, Proposition~4.15,
  it has been assumed implicitly that~$\kaphat=0$.
  More precisely,
  Neumann introduces a ``transfer relation''
	$$[z;p,q]+[z;p',q']=[z;p,q']+[z;p,q']
	\qquad\text{for all~$p$, $q$, $p'$, }q'\in\Z$$
  in the definition of his extended Bloch group.
  The analogous relation in our context would read
	$$[z;2p,2q]+[z;2p',2q']=[z;2p,2q']+[z;2p,2q']
	\qquad\text{for all~$p$, $q$, $p'$, }q'\in\Z\;.$$
  In analogy with Proposition~7.2 in~\cite{N},
  one can show that the effect of the transfer relation above
  is equivalent to dividing by the subgroup of order two
  that is generated by~$\kaphat$.

  We have just computed~$\Lhat(\kaphat)=-2\pi^2\in\C/4\pi^2\Z$
  in~\eqref{LhatKaphatFormel}.
  This explains that if one includes the transfer relation,
  then~$\Lhat$ is well-defined only modulo~$2\pi^2\Z$ as in~\cite{DZ}.

  Assuming that~$\kaphat=0$,
  one finds that~$\{z;2p\}$ becomes independent of~$p$.
  This allows to define a homomorphism~$\C^\times\to\Phat(\C)/\<\kaphat\>$
  with~$z\mapsto\{z;0\}$, see~\cite{N}, Proposition~8.2.
\end{Remark}

Starting from Neumann's map,
we obtain a pullback square
\begin{equation*}
  \begin{CD}
    \C^\times@>\chihat>>\Phat(\C)\\
    @VVV@VVV\\
    \C^\times@>>>\Phat(\C)/\<\kaphat\>&\;.
  \end{CD}
\end{equation*}
Here the left vertical arrow maps~$z$ to~$z^2$,
and~$\chihat$ is given by
\begin{equation}\label{ChihatDef}
  \chihat(z)=
  \begin{cases}
	0		&\text{if~$z=1$,}\\
	\kaphat	&\text{if~$z=-1$,}\\
	\{z^2+0i;0\}	&\text{if~$\Arg z\in\bigl(-\frac\pi2,\frac\pi2\bigr]$
				and~$z\ne 1$, and}\\
	\{z^2+0i;2\}	&\text{if~$\Arg z
				\notin\bigl(-\frac\pi2,\frac\pi2\bigr]$
				and~$z\ne -1$.}\\
  \end{cases}
\end{equation}

\begin{Theorem}\label{ChihatLemma}
  The map~$\chihat$ is a homomorphism,
  and the sequence
  \begin{equation*}
    \begin{CD}
      0@>>>\C^\times@>\chihat>>\Phat(\C)@>>>\P(\C)@>>>0
    \end{CD}
  \end{equation*}
  is exact and split,
  where the right arrow
  is induced by~$\pihat\colon\Chat\to\C\setminus\{0,1\}$.
\end{Theorem}

\begin{proof}
  First of all we note that by the definition of~$\kaphat$
  and Lemma~\ref{KappaLemma},
  \begin{equation*}
    \{z^2;2p\}+\kaphat=\{z^2;2p+2\}=\{z^2;2p-2\}\;,
  \end{equation*}
  which implies that~$\chihat(z)+\chihat(-1)=\chihat(-z)$,
  and that~$\{z^2,2p+4\}=\{z^2;2p\}$.
  By Lemma~\ref{HomoLemma},
  we have~$\chihat(z)+\chihat(w)=\chihat(zw)$
  for almost all choices of~$z$, $w\in\C^\times$.
  The remaining cases are easily checked.

  The right arrow maps~$\{z;2p\}=[z;2p,2]-[z;2p,0]$
  to~$[z]-[z]=0$, so the sequence above is a chain complex.
  To prove injectivity of~$\chihat$
  consider the composition~$\Lhat\circ\chihat$.
  For~$z\in\C^\times$
  let~$p=0$ if~$\Arg z\in\bigl(-\frac\pi2,\frac\pi2\bigr]$
  and~$p=1$ if~$\Arg z\notin\bigl(-\frac\pi2,\frac\pi2\bigr]$.
  Then
	$$\Log(z^2+0i)+2\pi i\,p\equiv 2\Log z\qquad\text{modulo }2\pi i\Z\;,$$
  and
  \begin{align}
    \begin{split}\label{SplitEq}
    (\Lhat\circ\chihat)(z)
    &=\frac12\,\Bigl((\Log z^2+2\pi i\,p)(\Log(1-z^2)+2\pi i)
	-(\Log z^2+2\pi i\,p)\Log(1-z^2)\Bigr)\\
    &=2\pi i\,\Log z\quad\in\quad\C/\Z(2)\;,
    \end{split}
  \end{align}
  and this even holds for~$z=\pm 1$, cf.~\eqref{LhatKaphatFormel},
  hence~$\chihat$ is injective.

  It remains to show that~$\im\chihat=\ker(\Phat(\C)\to\P(\C))$.
  Relations~\eqref{QRel} and~\eqref{PRel} allow to represent each generator
  of~$\Phat(\C)$ as
  \begin{equation*}
    [z;2p,2q]
    =pq\,[z;2,2]-p(q-1)\,[z;2,0]-(p-1)q\,[z;0,2]+(p-1)(q-1)\,[z;0,0]\;,
  \end{equation*}
  see~\cite{N}, Lemma~7.1.
  Using~\eqref{MirrorRel},
  we see that
  the kernel of~$\Phat(\C)\to\P(\C)$ is generated by elements of the form
  \begin{align*}
    [z;2p,2q]-[z;0,0]
    &=pq\,\{z;2\}-(p-1)q\,\{z;0\}+p\,\{1-z;0\}
    \in\im\chihat\;.
  \end{align*}

  By~\eqref{SplitEq},
  a splitting of the sequence is given by the
  homomorphism~$\exp\circ\frac{\Lhat}{2\pi i}\colon\Phat(\Z)\to\C^\times$.
\end{proof}

\begin{Corollary}\label{ExtBlochCor}
  The sequence 
  \begin{equation*}
    \begin{CD}
      0@>>>\Q/\Z@>\alpha\mapsto\chihat(e^{2\pi i\,\alpha})>>\Bhat(\C)
      @>>>\B(\C)@>>>0
    \end{CD}
  \end{equation*}
  is exact,
  and
  \begin{equation*}
    \frac1{(2\pi i)^2}\,\Lhat\bigl(\chihat(e^{2\pi i\,\alpha})\bigr)
    =\alpha\;.
  \end{equation*}
\end{Corollary}

\section{The Cheeger-Chern-Simons Class
\texorpdfstring{and~$H_3(\SL(2,\C^\delta))$}{}}\label{CCSSect}
Recall that a map~$\lamhat\colon H_3(\SL(2,\C^\delta))\to\Bhat(\C)$
has been constructed in~\cite{DZ}, Section~3,
without using the transfer relation.
Following~\cite{DZ},
we prove that~$\Lhat\circ\lamhat=\chat_2\in\C/\Z(2)$
and conclude from this that~$\lamhat$ is an isomorphism.

Note that because~$\C/\Z(2)$ is divisible,
there is a canonical isomorphism
\begin{equation*}\label{UCoeffThm}
  H^3\bigl(\SL(2,\C^\delta),\C/\Z(2)\bigr)
  \cong\Hom_\Z\bigl(H_3(\SL(2,\C^\delta)),\C/\Z(2)\bigr)\;.
\end{equation*}
Let~$\chat_2\in H^3(\SL(2,\C^\delta),\C/\Z(2))$ denote
the second Cheeger-Chern-Simons class of the tautological flat complex
vector bundle of rank~$2$ over~$B\SL(2,\C^\delta)$.
Here,
we are using the same normalisation as~\cite{N}.
In~\cite{DZ},
the class~$\frac1{(2\pi i)^2}\,\chat_2\in H^3(\SL(2,\C^\delta),\C/\Z)$
is considered, see~\cite{DZ}, Remark~4.2.

\begin{Theorem}[cf.~\cite{DZ}, Theorem~4.1]\label{MainThm}
  Under the isomorphism above,
	$$\chat_2=\Lhat\circ\lamhat\;.$$
\end{Theorem}

\begin{proof}
  By Theorem~4.1 in~\cite{DZ},
  we have that
	$$2\chat_2=2\Lhat\circ\lamhat
	\in\Hom_\Z\bigl(H_3(\SL(2,\C^\delta)),\C/\Z(2)\bigr)$$
  in our normalisation.
  Because~$H_3(\SL(2,\C^\delta))$ is divisible,
  this implies our claim.
\end{proof}

\begin{Corollary}[cf.~\cite{DZ}, Theorem~4.15]\label{MainCor}
  The map~$\lamhat\colon H_3(\SL(2,\C^\delta))\to\Bhat(\C)$
  is an isomorphism.
\end{Corollary}

\begin{proof}
  Consider the commutative diagram
  \begin{equation*}
    \begin{CD}
      0@>>>\Q/\Z@>\iota>>H_3(\SL(2,\C^\delta))@>\lambda>>\B(\C)@>>>0\\
      &&@|@VV\lamhat V@|\\
      0@>>>\Q/\Z@>\chihat(e^{2\pi i\punkt})>>\Bhat(\C)@>>>\B(\C)@>>>0&\;.
    \end{CD}
  \end{equation*}
  The upper row has been established in~\cite{DS} and is exact.
  The lower row is just Corollary~\ref{ExtBlochCor}.
  Commutativity of the right hand square has been established in~\cite{DZ},
  Section~3.
  
  Let~$\alpha\in\Q$,
  then~$\frac1{(2\pi i)^2}\,(\chat_2\circ\iota)(\alpha)=\alpha$
  by~\cite{Du2}, Theorem~10.2.
  Theorem~\ref{MainThm} implies that
	$$\frac1{(2\pi i)^2}\,\Lhat\bigl((\lambda\circ\iota)(\alpha)\bigr)
	=\frac1{(2\pi i)^2}\,(\chat_2\circ\iota)(\alpha)=\alpha\;,$$
  and~$(\lambda\circ\iota)(\alpha)\in\ker\bigl(\Bhat(\C)\to\B(\C)\bigr)$
  by commutativity of the right hand square.
  On the other hand,
  Corollary~\ref{ExtBlochCor} implies that
	$$\frac1{(2\pi i)^2}\,\Lhat\bigl(\chihat(e^{2\pi i\alpha})\bigr)
	=\alpha\;,$$
  and that~$\Lhat$ is injective on~$\ker\bigl(\Bhat(\C)\to\B(\C)\bigr)$.
  Thus the left hand square also commutes.
  Our claim now follows from the five-lemma.
\end{proof}

\begin{Remark}\label{PSLRem}
  Let $\Bhat_N(\C)\cong H_3(PSL(2,\C))$ denote Neumann's
  extended Bloch group in~\cite{N}.
  Then the diagramme
  \begin{equation*}
    \begin{CD}
      &&0&&0\\
      &&@VVV@VVV\\
      &&\Z/4\Z@=\Z/4\Z\\
      &&@VVV@V\chihat(i)\punkt VV\\
      0@>>>\Q/\Z@>\chihat(e^{2\pi i\punkt})>>\Bhat(\C)@>>>\B(\C)@>>>0\\
      &&@V4\punkt VV@VbVV@|\\
      0@>>>\Q/\Z@>\chi(e^{2\pi i\punkt})>>\Bhat_N(\C)@>>>\B(\C)@>>>0\\
      &&@VVV@VVV\\
      &&0&&0
    \end{CD}
  \end{equation*}
  commutes and has exact rows and columns.
  Here the map~$b\colon\Bhat(\C)\to\Bhat_N(\C)$ sends a generator~$[z,2p,2q]$
  to the same generator in~$\Bhat_N(\C)$,
  and~$\chi$ has been defined in~\cite{N}, Proposition~7.4.
  This is proved in analogy with Corollary~8.3 in~\cite{N}.
  For example,
  commutativity of the lower left hand square follows from
  \begin{equation*}
    (b\circ\chihat)(z)=[z^2;2p,2]-[z^2;2p,0]
    =2\bigl([z^2;2p,1]-[z^2;2p,0]\bigr)=2\chi(z^2)=4\chi(z)
    \in\Phat_N(\C)\;.
  \end{equation*}
  This also shows that~$\ker b$ is spanned by~$\chihat(i)=\{-1;0\}$.
\end{Remark}

\section{More Relations in the Extended pre-Bloch Group}\label{MoreSect}
By~\cite{DS},
one has the relations
	$$[z]=-[1-z]=\biggl[\frac1{1-z}\biggr]
	=-\biggl[-\frac z{1-z}\biggr]=\biggl[1-\frac1z\biggr]
	=-\biggl[\frac1z\biggr]\;.$$
in the pre-Bloch group~$\P(\C)$.
If we interpret~$z$ as the cross-ratio of a generic configuration
of four points in~$\C P^1$,
then these relations say that up to orientation,
the order of the points is not important. 
Note that~$\Im\chat_2$ is already well-defined on~$\B(\C)$.

Similar relations hold in Neumann's extended pre-Bloch group~$\Phat_N(\C)$
by Proposition~13.1 in~\cite{N}.
As a consequence,
unordered oriented simplices are also sufficient to compute~$\Re\chat_2$ up
to some finite ambiguity.
Unfortunately,
these relations become more complicated in~$\Phat(\C)$.
Let~$\root{4\,}\of z$ denote the standard fourth root of~$z$
with~$\Arg\root{4\,}\of z\in\bigl(-\frac\pi4,\frac\pi4\bigr]$.

\begin{Proposition}\label{RelThm}
  Let~$\Im z>0$. Then
  \allowdisplaybreaks
  \begin{align*}
    \biggl[\frac1z;-2p,2p+2q\biggr]
    &=-[z;2p,2q]+\chihat\bigl(i^p\root{4\,}\of z\bigr)\;,\tag1\\
    \biggl[1-\frac1z;-2p-2q;2p\biggr]
    &=[z;2p,2q]-\chihat\Bigl(e^{-\tfrac{\pi i}{12}(1-6p)}
	\root{4\,}\of z\Bigr)\;,\tag2\\
    \biggl[-\frac z{1-z};2p+2q,-2q\biggr]
    &=-[z;2p,2q]+\chihat\Bigl(e^{-\tfrac{\pi i}{12}\,(1+6q)}
	\,\root{4\,}\of{z-1}\Bigr)\;,\tag3\\
    \biggl[\frac1{1-z};2q,-2p-2q\biggr]
    &=[z;2p,2q]-\chihat\Bigl(e^{-\tfrac{\pi i}{12}\,(2+6q)}
	\,\root{4\,}\of{z-1}\Bigr)\;,\tag4\\
    [1-z;-2q,-2p]
    &=-[z;2p,2q]+\chihat\Bigl(e^{\tfrac{\pi i}{12}}\Bigr)\;.\tag5
  \end{align*}
\end{Proposition}

\begin{proof}
  The involutions
	$$[z;2p,2q]\mapsto[1-z;-2q,-2p]
		\qquad\text{and}\qquad
	[z;2p,2q]\mapsto\biggl[\frac 1z;-2p;2p+2q\biggr]$$
  generate an action of~$S(3)$ on~$\Chat$.
  Using~$\Im\bigl(1-\frac1z\bigr)>0$ and~$\Im\frac1{1-z}>0$,
  it is now easy to see that the five relations above follow from~(1) and~(5).
  Note that by~\cite{DS},
  both relations are true
  modulo~$\ker\bigl(\Phat(\C)\to\P(\C)\bigr)=\im\chihat$.
  Because the Rogers dilogarithm~$\Lhat$ is injective
  on~$\im\chihat$ by~\eqref{SplitEq},
  it suffices to check both relations after applying~$\Lhat$.
  This can be done using some elementary facts about the classical
  dilogarithm, and is thus left to the reader.
%
\end{proof}

\begin{Remark}\label{NeumannRelRemark}
  The relations in~\cite{N} look somewhat nicer,
  since the correction term does not involve the variable~$z$.
  This is possible because in Neumann's definition of~$\Phat_N(\C)$,
  odd integers are allowed,
  so that one can consider the involution~$[z;p,q]\mapsto[1/z;-p,1+p+q]$.

  Following~\cite{N},
  the various terms on the left hand side of the equations
  in Proposition~\ref{RelThm} correspond to flattenings of a given
  oriented simplex with different orderings of vertices.
  Thus the proposition seems to indicate that it is not enough
  to consider unordered oriented simplices
  if one wants to compute the full class~$\chat_2$.

  For higher classes~$\chat_k$,
  not many formulas are available.
  The only formula known to the authors is the formula for~$\Re\chat_3$
  in~\cite{Goncharov},
  which uses unordered oriented simplices.
\end{Remark}


\end{document}